\documentclass[11pt]{amsart}
\usepackage{amssymb,bm,mathrsfs}
\usepackage{mathabx}
\usepackage[all]{xy}
\usepackage{xcolor}
\usepackage{hyperref}
\hypersetup{colorlinks=true}

\newcommand{\map}[1]{\xrightarrow{#1}}
\newcommand{\mil}{\varprojlim}

\newcommand{\iso}{\cong}

\newcommand{\Hom}{\mathrm{Hom}}

\newcommand{\End}{\mathrm{End}}

\newcommand{\Q}{\mathbb Q}
\newcommand{\Z}{\mathbb Z}
\newcommand{\R}{\mathbb R}

\newcommand{\F}{\mathbb F}

\newcommand{\co}{\mathcal O}

\newcommand{\Lie}{\mathrm{Lie}}

\newcommand{\ord}{\mathrm{ord}}

\newcommand{\Ainf}{A_\mathrm{inf}}
\newcommand{\HT}{\mathrm{HT}}
\newcommand{\et}{\mathrm{et}}
\newcommand{\crys}{\mathrm{crys}}
\newcommand{\dR}{\mathrm{dR}}

\begin{document}
\author{Benjamin Howard}
\title{On the  Ekedahl-Oort stratification of Shimura curves}
\date{}
\thanks{This research was supported in part by NSF grants DMS1501583 and DMS1801905.}
\address{Department of Mathematics\\Boston College\\ 140 Commonwealth Ave. \\Chestnut Hill, MA 02467, USA}
\email{howardbe@bc.edu}

\begin{abstract}
We study the Hodge-Tate period domain associated to a quaternionic Shimura curve at a prime of bad reduction, 
and give an explicit description of its Ekedahl-Oort stratification.
\end{abstract}

\maketitle


\theoremstyle{plain}
\newtheorem{theorem}{Theorem}[subsection]
\newtheorem{bigtheorem}{Theorem}[section]
\newtheorem{proposition}[theorem]{Proposition}
\newtheorem{lemma}[theorem]{Lemma}
\newtheorem{corollary}[theorem]{Corollary}

\theoremstyle{definition}
\newtheorem{definition}[theorem]{Definition}
\newtheorem{bigdefinition}[bigtheorem]{Definition}

\theoremstyle{remark}
\newtheorem{remark}[theorem]{Remark}
\newtheorem{example}[theorem]{Example}
\newtheorem{question}[theorem]{Question}
\newtheorem*{hypothesis}{Hypothesis}

\numberwithin{equation}{subsection}
\renewcommand{\thebigtheorem}{\Alph{bigtheorem}}
\renewcommand{\thebigdefinition}{\Alph{bigdefinition}}


\section{Introduction}


Fix a prime $p$, and let $C$ be the completion of an algebraic closure of $\Q_p$.
Denote by $\co\subset C$ its ring of integers, and by  $k=\co/\mathfrak{m}$ its residue field.


\subsection{Stratifications of $p$-adic periods domains}


Let  $G$ be a $p$-divisible group over $\co$.   It has a   $p$-adic Tate module
\[
T_p(G) = \Hom( \Q_p / \Z_p , G) 
\] 
and a module of invariant differential forms $\Omega(G)$.  These are  free of finite rank over $\Z_p$ and $\co$, respectively.
Using the canonical trivialization $\Omega(\mu_{p^\infty}) \iso \co$, we   define the   \emph{Hodge-Tate morphism}
\begin{equation}\label{HT}
T_p(G) \iso \Hom( G^\vee , \mu_{p^\infty}) \map{\HT} \Hom(  \Omega(\mu_{p^\infty}) , \Omega(G^\vee)  ) \iso  \Omega(G^\vee),
\end{equation}
where $G^\vee$ is the $p$-divisible group dual to $G$.

\begin{bigtheorem}[Scholze-Weinstein \cite{SW-moduli}]\label{thm:SW}
There is an equivalence between the  category of  $p$-divisible groups over $\co$  and the category of  pairs  $(T,W)$ in which  
\begin{itemize}
\item
$T$ is  a  free $\Z_p$-module of finite rank,  
 \item
$W\subset T\otimes_{\Z_p} C$   is a $C$-subspace. 
  \end{itemize}
The equivalence sends  $G$  to its $p$-adic Tate module $T=T_p(G)$,  endowed with its Hodge-Tate filtration 
\[
W = \mathrm{ker} \big(  T_p(G) \otimes_{\Z_p} C \map{\HT} \Omega(G^\vee)\otimes_\co C \big) .
\]
\end{bigtheorem}

Fix a free $\Z_p$-module $T$ of finite rank, and consider the $\Q_p$-scheme
\[
X=\mathrm{Gr}_d(T\otimes_{\Z_p} \Q_p)
\]
 parametrizing   subspaces of $T\otimes_{\Z_p} \Q_p$ of some  fixed dimension $d \le \mathrm{rank}(T)$.
  By the theorem of Scholze-Weinstein, every point $W\in X(C)$
determines a $p$-divisible group $G$ over $\co$, whose reduction to the residue field we denote by $G_k$.  Let $G_k[p]$ be the group scheme 
of $p$-torsion points in $G_k$.

If we declare  two points $W,W'\in X(C)$ to be equivalent when the corresponding reductions $G_k$ and $G'_k$  are isogenous,  the resulting partition   is the \emph{Newton stratification} of $X(C)$.
Alternatively, if we  declare  $W,W'\in X(C)$ to be equivalent when the $p$-torsion group schemes  $G_k[p]$ and  $G'_k[p]$ are isomorphic,   the resulting partition is the \emph{Ekedahl-Oort stratification} of $X(C)$.  

There are similar partitions when  $X$ is replaced by a more sophisticated flag variety, called the \emph{Hodge-Tate period domain}, associated to a Shimura datum of Hodge type and a prime $p$.  This period domain and its Newton stratification were  studied by Caraiani-Scholze \cite{CS}, who proved  that each Newton stratum in $X(C)$ can be realized as the $C$-points of a locally closed subset of the associated adic space. 
For the  Ekedahl-Oort stratification of $X(C)$ there is nothing in the existing literature, and it is not known if it has any structure other than  set-theoretic partition.   

In the case of modular curves, the Hodge-Tate period domain is the projective line $\mathbb{P}^1$ over $\Q_p$.  In this case the  Newton stratification and the Ekedahl-Oort stratification agree, and there are two strata:
the \emph{ordinary locus}   $\mathbb{P}^1(\Q_p)$,  and  the \emph{supersingular locus}  $\mathbb{P}^1(C) \smallsetminus \mathbb{P}^1(\Q_p)$.  

For  the compact Shimura curve determined by an indefinite quaternion algebra over $\Q$, and a prime $p$ at which the quaternion algebra is ramified,   the Hodge-Tate period domain $X$ is a twisted form of $\mathbb{P}^1$.  All points of $X(C)$ give rise to supersingular $p$-divisible groups over $k$, and  the Newton stratification consists of a  single stratum, $X(C)$ itself.  In contrast, the Ekedahl-Oort stratification is nontrivial, and the goal of this paper is to make it explicit.

Although the  methods used here are fairly direct,  it is not clear  how far they can be extended.
The case of Hilbert  modular surfaces may already require new ideas.

For background on the  classical  Ekedahl-Oort stratification of  reductions of  Shimura varieties (as opposed to their Hodge-Tate period domains), we refer the reader to the  work of Oort \cite{Oort01},  Moonen \cite{Moonen}, Viehmann-Wedhorn \cite{VW},  Zhang \cite{Zhang}, and the references found therein.


\subsection{The Shimura curve period domain}


 Let $\Q_{p^2}\subset C$ be the unique unramified quadratic extension of $\Q_p$, and let  $\Z_{p^2} \subset \co$ be its ring of integers.  
 Denote by   $x\mapsto \overline{x}$ the nontrivial automorphism of $\Q_{p^2}$.  Define a  non-commutative  $\Z_p$-algebra of rank $4$ by
\[
\Delta=\Z_{p^2}  [\Pi]  ,
\]
where $\Pi$ is subject to the relations  $\Pi^2=p$ and $\Pi\cdot x= \overline{x} \cdot  \Pi$ for all $x\in \Z_{p^2}$.  
In other words,  $\Delta$ is the unique maximal order in the unique quaternion  division algebra over $\Q_p$.

Let $T$ be a free $\Delta$-module of rank one, and let $X$ be the smooth projective variety over   $\Q_p$  with functor of points 
\begin{equation}\label{drinfeld HT}
X(S) = \left\{  \begin{array}{c} 
\mbox{$\co_S$-module local direct summands} \\ W \subset T\otimes_{\Z_p} \co_S \\ \mbox{of rank $2$  that are stable under $\Delta$}
\end{array} \\  \right\} 
\end{equation}
for any $\Q_p$-scheme $S$.   This is the Hodge-Tate period domain associated to a quaternionic  Shimura curve.

 As we explain in \S \ref{ss:main setup},   our period domain becomes isomorphic to the projective line after base change to $\Q_{p^2}$, and any choice of $\Delta$-module generator $\lambda \in T$  determines a bijection
\begin{equation}\label{intro coordinates}
  X(C) \iso C \cup \{ \infty\} .
\end{equation}
After fixing such a choice, we normalize  the valuation $\ord: C \to \R\cup \{\infty\}$ by $\ord(p) =1$, extend it to $C\cup \{\infty\}$ by $\ord(\infty)=-\infty$, and use (\ref{intro coordinates}) to view $\ord$ as a function
\[
\ord : X(C) \to  \R \cup\{ -\infty ,\infty \}.
\]

The theorem of Scholze-Weinstein  provides a  canonical bijection
\[
X(C) \iso 
 \left\{  \begin{array}{c}
\mbox{isomorphism classes of $p$-divisible groups $G$ over $\co$ } \\ \mbox{of height $4$ and dimension $2$, endowed  with an action} \\ \mbox{of $\Delta$ and a $\Delta$-linear isomorphism $T_p(G) \iso T$ }
\end{array} \right\} .
\]
By forgetting the level structure $T_p(G) \iso T$, reducing to the residue field, and then taking $p$-torsion subgroups, we obtain a function
\[
X(C) \to  \left\{  \begin{array}{c}
\mbox{isomorphism classes of finite group schemes} \\ \mbox{over $k$, endowed with an action of $\Delta/p\Delta$} 
\end{array} \right\}  
\]
sending $G\mapsto G_k[p]$, whose   fibers   are the \emph{Ekedahl-Oort strata} of $X(C)$.

\begin{hypothesis}
    For the rest of this introduction, we assume $p>2$.  Theorems \ref{bigone} and \ref{bigtwo} below are presumably true without this hypothesis, but we are unable to provide a proof.  See the remarks following Theorem \ref{thm:BKF}.
    \end{hypothesis}

It is convenient to organize the strata into two types:  those on which the $p$-torsion group scheme $G_k[p]$ is superspecial (in the sense of \S \ref{ss:reduction}), and those on which it is not.
The two theorems that follow show that there are three superspecial strata, and two infinite families of non-superspecial strata.  These results are proved in  \S \ref{ss:main theorems}, where the reader will also find an explicit recipe for computing   the Dieudonn\'e module of the  $p$-torsion group scheme $G_k[p]$ attached to a point of $X(C)$.

\begin{bigtheorem}\label{bigone}
The conditions
\[
\frac{1}{p+1} < \ord(\tau) < \frac{p}{p+1}
\]
on  $\tau \in  X(C)$ define an Ekedahl-Oort stratum, as do each one of 
\[
\ord(\tau) < \frac{1}{p+1}  , \qquad 
 \frac{p}{p+1}  <  \ord(\tau)  .
\]
The union of these three strata is the locus of points with superspecial reduction.  In particular, 
 the isomorphism class of the finite  group scheme $G_k[p]$  is the same all on three strata, but the isomorphism class of $G_k[p]$ with its  $\Delta$-action  is not.
  \end{bigtheorem}

Now consider the locus of points 
\begin{equation}\label{nssp}
\left\{ \tau \in C : \ord(\tau) = \frac{1}{p+1} \right\}  \cup \left\{ \tau \in C : \ord(\tau) = \frac{p}{p+1} \right\}
\subset X(C)
\end{equation}
at which the corresponding $p$-divisible group does not have superspecial reduction.  
The isomorphism class of the $p$-torsion group scheme $G_k[p]$ is constant on (\ref{nssp}), but the isomorphism class of $G_k[p]$ with its $\Delta$-action varies. 
 In fact, the $\Delta$-action  varies so much that  (\ref{nssp}) decomposes as an infinite disjoint union of Ekedahl-Oort strata.

\begin{bigtheorem}\label{bigtwo}
The fibers of the composition
\[
\left\{ \tau \in C : \ord(\tau) = \frac{1}{p+1} \right\} \map{ \tau \mapsto p/\tau^{p+1}} \co^\times \to k^\times  
\]
are Ekedahl-Oort strata, as are the fibers of the composition
\[
\left\{ \tau \in C : \ord(\tau) = \frac{p}{p+1} \right\} \map{ \tau \mapsto \tau^{p+1}/p^p} \co^\times \to k^\times .
\]
Both unlabeled arrows are reduction to the residue field.
\end{bigtheorem}

\begin{remark}
The infinitude of Ekedahl-Oort  strata is a pathology arising from the non-smooth reduction of  compact Shimura curves.
Similar pathologies  for the reductions of Hilbert modular varieties at  ramified primes are described in the appendix to \cite{AG}.
\end{remark}


\subsection{Notation and conventions}


 Throughout the  paper $p$ is a fixed prime. We allow $p=2$ unless otherwise stated. 
 Let $k=\co/\mathfrak{m}$ as above, and denote by $\sigma : k \to k$ the absolute Frobenius $\sigma(x)=x^p$. 
 
 The rings $\Z_{p^2} \subset \co$  and  $\Delta=\Z_{p^2}[\Pi]$ have the same meaning as  above.  We label the embeddings
 \begin{equation}\label{embeddings}
 j_0, j_1: \Z_{p^2} \to \co
 \end{equation}
in such a way that $j_0$ is the inclusion and  $j_1(x) = j_0(\overline{x})$ is its conjugate.


\subsection{Acknowledgements}


The author thanks Keerthi Madapusi Pera for helpful conversations, and the anonymous referees for numerous suggestions and corrections.


\section{Integral $p$-adic Hodge theory}


In this section we recall the integral $p$-adic Hodge theory of an  arbitrary   $p$-divisible group $G$ over $\co$.
The quaternion order $\Delta$ plays no role whatsoever.

Following \cite{fargues},  \cite{Snotes},  and \cite{lau-BKF}, we will attach to $G$ a Breuil-Kisin-Fargues module,
and explain how to extract from it invariants of $G$ such as its Hodge-Tate morphism $T_p(G) \to \Omega(G^\vee)$, and the Dieudonn\'e module of its reduction to $k$.


\subsection{A ring of periods}
\label{ss:tilt}


Let $C^\flat$ be the tilt of $C$, with ring of integers $\co^\flat$.  Thus
\[
\co^\flat =   \mil_{ x\mapsto x^p} \co/(p) 
\]
is a local  domain of characteristic $p$, fraction field  $C^\flat$, and residue field $k=\co^\flat/\mathfrak{m}^\flat$.
An element $x\in \co^\flat$ is given by a sequence $(x_0,x_1,x_2,\ldots)$ of elements  $x_\ell\in \co/(p)$ satisfying $x^p_{\ell+1}=x_\ell$.
After choosing  arbitrary lifts  $x_\ell\in\co$,  set 
\[
x^\sharp = \lim_{\ell\to \infty} x_\ell^{p^\ell}.
\]
The construction $x\mapsto x^\sharp$ defines a multiplicative function $\co^\flat \to \co$, and we define 
 $\ord: \co^\flat \to \R\cup \{\infty\}$ by $\ord(x) =\ord(x^\sharp)$.

Denote by $\sigma: \co^\flat \to \co^\flat$ the absolute Frobenius $x\mapsto x^p$, and in the same way the induced automorphism of the local domain
\[
\Ainf=W(\co^\flat).
\]
There is a canonical homomorphism of $\Z_p$-algebras
\[
\Theta : \Ainf \to \co
\]
satisfying $\Theta([x]) = x^\sharp$, where $[\cdot] : \co^\flat \to \Ainf$ is the Teichmuller lift.

The kernel of $\Theta$ is a principal ideal. 
To construct a generator,  first fix a $\Z_p$-module generator
\[
\zeta = (\zeta_p,\zeta_{p^2},\zeta_{p^3}, \ldots)  \in T_p(\mu_{p^\infty})
\]
 and define
$
\epsilon =( 1 ,\zeta_p,\zeta_{p^2},\ldots) \in \co^\flat.
$
The element 
\[
\xi  = [1]+[\epsilon^{1/p}]+[\epsilon^{2/p}]+ \cdots+ [\epsilon^{(p-1)/p} ] \in \Ainf
\]
generates $\mathrm{ker}(\Theta)$.  If we denote by 
\[
\varpi = 1+\epsilon^{1/p}+\epsilon^{2/p}+ \cdots+ \epsilon^{ (p-1)/p}  \in \co^\flat 
\]
its image under the reduction map $\Ainf \to \Ainf/(p) = \co^\flat$, then $\ord(\varpi)=1$, and   there are  canonical isomorphisms
\[
\co/(p) \iso \Ainf/(\xi,p) \iso \co^\flat/(\varpi).
\]

The following lemma will be needed in the proof of Proposition \ref{prop:second case}.

\begin{lemma}\label{lem:sharp reduction}
The reduction map $\co^\times \to k^\times$ sends $\varpi^\sharp / p  \mapsto -1$.
\end{lemma}

\begin{proof}
By definition,
$
\varpi^\sharp = \lim_{\ell \to \infty} x_\ell^{p^\ell}
$
where
\[
x_\ell =  1 + \zeta_{p^{\ell+1}} + \zeta_{p^{\ell+1}}^2 +\cdots +  \zeta_{p^{\ell+1}}^{p-1}
= \frac{\zeta_{p^\ell} -1 }{\zeta_{p^{\ell+1}}  -1 } \in \co.
\]
The binomial theorem implies that
\[
\zeta_{p^\ell}  = ( \zeta_{p^{\ell+1}} -1 +1)^p    = ( \zeta_{p^{\ell+1}} -1 )^p  + s p ( \zeta_{p^{\ell+1}} -1 )   +1 
\]
for some $s\in \co$.  From this we deduce first that 
\[
x_\ell \equiv  ( \zeta_{p^{\ell+1}} -1 )^{p-1} \pmod{p\co},
\]
and then that
\begin{equation}\label{cyclotomic uniformizer}
x_\ell^{p^\ell}  \equiv ( \zeta_{p^{\ell+1}} -1 )^{(p-1)p^\ell}   \pmod{p^{\ell-1}\co}.
\end{equation}

For $1\le i \le p-1$ set 
\[
u_i  =     \frac{ 1- \zeta_p^i}{1-\zeta_p}  = 1+ \zeta_p+ \cdots + \zeta_p^{i-1} \in \co^\times,
\]
and note that Wilson's theorem implies 
$
u_1\cdots u_{p-1} \equiv -1\pmod{\mathfrak{m}}.
$
Taking $X=1$ in the factorization
\[
X^{p-1}+ \cdots+ X +1 = ( X-\zeta_p)\cdots(X-\zeta_p^{p-1})
\]
shows that 
$
p= (1-\zeta_p)^{p-1} u_1\cdots u_{p-1},
$
and hence
\[
\frac{(1-\zeta_p)^{p-1}}{p} \equiv  -1 \pmod{\mathfrak{m}}.
\]
Combining this with (\ref{cyclotomic uniformizer}) shows that
\[
\frac{x_\ell^{p^\ell}}{p} 
\equiv  \frac{ ( \zeta_{p^{\ell+1}} -1 )^{(p-1)p^\ell}}{p}
\equiv -   \left( \frac{ ( \zeta_{p^{\ell+1}} -1 )^{p^\ell}}{ 1-\zeta_p } \right)^{p-1} \pmod{\mathfrak{m}}.
\]

As  $\Q_p( \zeta_{p^{\ell+1}})$ is totally ramified over $\Q_p$, the reduction of 
\[
 \frac{ ( \zeta_{p^{\ell+1}} -1 )^{p^\ell}}{ 1-\zeta_p }  \in \co^\times
\]
lies in the subgroup $\F_p^\times \subset k^\times$. 
 It follows that $x_\ell^{p^\ell} / p  \equiv -1 \pmod{\mathfrak{m}}$, and hence 
$\varpi^\sharp/p\equiv -1\pmod{\mathfrak{m}}$.
\end{proof}


\subsection{Breuil-Kisin-Fargues modules}
\label{ss:BKF}


There is an equivalence of categories between $p$-divisible groups over $\co$ and Breuil-Kisin-Fargues modules, whose definition we now recall.

\begin{definition}
A \emph{Breuil-Kisin-Fargues module} is a triple $(M,\phi,\psi)$, in which $M$ is a  free module of finite rank over $\Ainf$,
and \[\phi ,\psi :M\to M\]  are  homomorphisms of additive groups satisfying 
\[
\phi(a m) = \sigma(a) \phi(m) ,\quad \psi( \sigma(a) m) = a\psi (m),
\]
for all $a\in \Ainf$ and $m\in M$, as well as $\phi \circ \psi = \xi .$
\end{definition}

Suppose $(M,\phi,\psi)$  is a Breuil-Kisin-Fargues module.
Denote by
\[
\sigma^* M  = \Ainf \otimes_{ \sigma , \Ainf } M
\] 
 the Frobenius twist of $M$, and   by $N$ the image of the $\Ainf$-linear map
\[
 M  \map{ x\mapsto 1\otimes \psi(x)} \sigma^*M.
\] 
It is easy to see that $\xi \sigma^*M \subset N \subset \sigma^*M$.
We construct  various realizations of $M$ as follows:
\begin{itemize}
\item
The \emph{de Rham realization} 
\[
M_\dR  = \sigma^*M / \xi \sigma^*M,
\]
 sits  in the short exact sequence  
\[
0 \to N / \xi\sigma^*M \to M_\dR \to \sigma^*M/N \to 0.
\]
of free $\co$-modules.  Indeed, the freeness of $M_\dR$ is clear, the freeness of $\sigma^*M/N$ follows from the proof of \cite[Lemma 9.5]{lau-BKF}, and the freeness of $N / \xi\sigma^*M$ is a  consequence of this.

\item
The \emph{\'etale realization} is the torsion-free $\Z_p$-module
\[
M_\et  = M^{\psi=1} .
\]
 Its \emph{Hodge-Tate filtration}
\[
F_\HT(M)  \subset  M_{\et} \otimes_{\Z_p} C 
\]
is  the kernel of the $C$-linear extension of 
\[
M_\et  \map{ x\mapsto 1\otimes \psi(x) }    N  /\xi \sigma^* M  . 
\]
\item
The \emph{crystalline realization}
\[
M_\crys  =  W(k) \otimes_{\sigma, \Ainf} M
\]
 is a free $W(k)$-module, endowed with operators
\[
F( a \otimes m) = \sigma(a) \otimes \phi(m) ,\qquad V( a\otimes m) = \sigma^{-1}(a) \otimes \psi(m).
\]
These give $M_\crys$ the structure of a Dieudonn\'e module.
\end{itemize}

The following theorem is no doubt known to the experts, but for lack of a  reference we will explain in the next subsection how to deduce it from the results of \cite{lau-BKF}.

\begin{theorem}[Fargues, Scholze-Weinstein, Lau]\label{thm:BKF}
Assume that $p>2$.
The category of $p$-divisible groups over $\co$ is equivalent to the category of  Breuil-Kisin-Fargues modules.
Moreover, the Breuil-Kisin-Fargues module  $(M,\phi,\psi)$ associated to a $p$-divisible group $G$ enjoys the following properties:
\begin{enumerate}
\item
There are isomorphisms of $\co$-modules
\begin{equation}\label{BKF-dR}
\Omega(G^\vee)  \iso N / \xi \sigma^*M ,\qquad \Lie(G) \iso \sigma^*M/ N.
\end{equation}
\item
If $G_k$ denotes the reduction of $G$ to the residue field $k=\co/\mathfrak{m}$, the covariant Dieudonn\'e module of $G_k$ is isomorphic to $M_\crys$.

\item
There is an isomorphism $T_p(G) \iso M_\et$ making the diagram
\begin{equation}\label{BKF-HT}
\xymatrix{
{ T_p(G) } \ar@{=}[rr] \ar[d]_{\HT}& &  {   M_\et } \ar[d]  \\ 
{ \Omega( G^\vee)   }   \ar@{=}[rr]  & & { N  /\xi \sigma^* M   }
}
\end{equation}
commute, where  the  vertical arrow on the right is the restriction to $M_\et \subset M$ of the $\co$-linear map 
\[
M \map{x\mapsto 1\otimes \psi(x) }  N /   \xi \sigma^*M.
\]
\end{enumerate}
All of these isomorphisms are functorial.
\end{theorem}

Some comments on this theorem are warranted, particularly regarding the restriction to $p>2$.
A  functor\footnote{Fargues only considers formal $p$-divisible groups, and imposes a corresponding restriction on Breuil-Kisin-Fargues modules.} from Breuil-Kisin-Fargues modules  to $p$-divisible groups over $\co$, but not a proof that it is an equivalence of categories,  first appeared in the work Fargues \cite[\S 4.8.1]{fargues}.  His construction of the functor makes essential use of the theory of \emph{windows}  introduced by Zink \cite{zink} and extended by Lau \cite{lau-frames, lau-BKF}, and assumes that $p>2$.

A proof of the  equivalence of categories\footnote{Our conventions for Breuil-Kisin-Fargues modules and  the equivalence of categories differ from those of  \cite{Snotes}.
The discrepancy amounts to a Tate twist.}    is found in  \cite[Theorem 14.1.1]{Snotes},
 where the result is attributed  to  Fargues.  
 The construction of the functor in \cite{Snotes} is very different from the  construction of \cite{fargues}, and does not use of the theory of windows.
 Instead, what is proved in \cite{Snotes} is that the category of Breuil-Kisin-Fargues modules is equivalent to the category of pairs $(T,W)$ appearing in Theorem \ref{thm:SW}, hence is equivalent to the category of $p$-divisible groups.   This proof comes with no restriction on $p$.

  The identification of $M_\crys$ with the Dieudonn\'e module of $G_k$ is \cite[Corollary 14.4.4]{Snotes}, and the isomorphism $T_p(G) \iso M_\et$ can be deduced  by carefully tracing through the construction of the equivalence.
  Unfortunately,  the isomorphisms of $\co$-modules (\ref{BKF-dR})  seem difficult to deduce from the description of the equivalence found in \cite{Snotes}.

Because of this, our equivalence of categories will be the one appearing in \cite{lau-BKF}, which follows  Fargues.
What  Lau proves is  that, when $p>2$,  the categories of Breuil-Kisin-Fargues modules and $p$-divisible groups over $\co$ are both equivalent to the category of windows.   
The various properties of the equivalence listed in Theorem \ref{thm:BKF}  can be read off  from the constructions
 of the two  functors into the category of windows, which are quite simple and direct (of course, the proof that they are equivalences is not).

The invocation of Theorem \ref{thm:BKF} in the calculations of \S \ref{s:core} is the only reason why   the assumption $p>2$ is imposed in the introduction.  Our approach in the sequel will  be to allow arbitrary $p$, but to take the conclusions of Theorem \ref{thm:BKF} as hypotheses.


\subsection{Proof of Theorem \ref{thm:BKF}}


As we have already indicated, Theorem \ref{thm:BKF} is proved by relating the categories of Breuil-Kisin-Fargues modules and $p$-divisible groups to the  category of windows introduced by Zink \cite{zink} and extended by Lau \cite{lau-frames, lau-BKF}.

Our windows will be modules over the ring $A_\crys$, which is defined as  the $p$-adic completion of the subring
\[
A_\crys^0=\Ainf[ \xi^n/n! : n=1,2,3,\ldots] \subset \Ainf[1/p].
\]
It  is an integral domain endowed with a  ring homomorphism 
\begin{equation}\label{theta crys}
\Theta_\crys : A_\crys \to \co
\end{equation}
extending $\Theta: \Ainf \to \co$, and divided powers on the kernel 
$
I = \ker (\Theta_\crys).
$

The  subring $A_\crys^0 \subset \Ainf[1/p]$ is stable under $\sigma$, and there is a unique continuous extension to an injective ring homomorphism  
$
\sigma : A_\crys \to A_\crys 
$
reducing to the usual $p$-power Frobenius on $A_\crys/pA_\crys$.  Moreover,
\cite[Lemma 4.1.8]{SW-moduli} and the comments following \cite[(9.1)]{lau-BKF} show that
\[
\sigma(I) \subset p A_\crys \quad \mbox{and}\quad  \frac{ \sigma(\xi)}{p}  \in A_\crys^\times .
\]

The following definition of a window is taken from  \cite[\S 2]{lau-BKF}, 
where it would be called a  \emph{window over the frame} 
\[
\underline{A}_\crys = (A_\crys ,  I ,  \co= A_\crys/I , \sigma ,\sigma_1),
\]  
with  $\sigma_1: I \to A_\crys$ defined by  $\sigma_1(x) = \sigma(x)/p$.

\begin{definition}
A \emph{window} is a quadruple $(P,Q,\Phi,\Phi_1)$ consisting of a projective $A_\crys$-module $P$ of finite rank, a submodule $Q\subset P$, and $\sigma$-semi-linear maps
\[
\Phi : P\to P ,\qquad \Phi_1 : Q\to P
\]
satisfying the following properties:
\begin{itemize}
\item 
there exist $A_\crys$-submodules $L,T\subset P$ such that  
\[
Q = L\oplus IT  ,\qquad P=L\oplus T,
\]
\item  $a\otimes x\mapsto a \Phi_1(x)$ defines a surjection $\sigma^*Q\to P$ of $A_\crys$-modules,
\item $\Phi(ax) = p \Phi_1( a x)$ for all $a\in I$ and $x\in  P$.
\end{itemize}
\end{definition}

\begin{remark}
Taking $a=\xi$ in the final condition yields
\[
\Phi( x) = \frac{p}{\sigma(\xi)}  \cdot  \Phi_1(  \xi x)
\]
for all $x\in P$.  This implies $\Phi(x) = p\Phi_1(x)$ for all  $x\in Q$, and  shows that each one of $\Phi$ and $\Phi_1$ determines the other.
\end{remark}

\begin{remark}
Note that   $IP \subset Q$, and that  $Q/IP$ and $P/Q$ are projective (hence free)  over $A_\crys/I \iso \co$.
\end{remark}

Suppose $G$ is a $p$-divisible group over $\co$.   
Let $P$ be  its crystalline Dieudonn\'e module, evaluated at  the divided power thickening (\ref{theta crys}).  
This is a projective $A_\crys$-module of rank equal to the height of $G$, equipped with a  $\sigma$-semi-linear operator $\Phi: P \to P$  and a short exact sequence
\[
0\to \Omega(G^\vee) \to P / I P \to \Lie(G) \to 0
\]
of free $\co$-modules.  Define  $Q\subset P$ as the kernel of  $P \to   \Lie(G).$
One can show that $\Phi(Q) \subset p P$, allowing us to define $\Phi_1 :  Q \to P$ by 
\[
\Phi_1( x) =  \frac{1}{p}  \cdot \Phi( x).
\]
 The following is a special case of \cite[Proposition 9.7]{lau-BKF}.

\begin{theorem}[Lau]\label{thm:group to window}
The  construction $G\mapsto (P,Q,\Phi,\Phi_1)$ just given defines a functor from the category of $p$-divisible groups over $\co$ to the category of windows.  It  is an equivalence of categories if $p>2$.
\end{theorem}

Now suppose we start with a Breuil-Kisin-Fargues module $(M,\phi,\psi)$.  
Set 
\begin{equation}\label{MtoP}
P = A_\crys \otimes_{\sigma,\Ainf} M,
\end{equation}
and  define  $\Phi: P \to P$  by 
$
\Phi(a \otimes m) = \sigma(a)  \otimes \phi(m)
$
for all $a\in A_\crys$ and $m\in M$.
The submodule $Q\subset P$, defined as the kernel of the composition
\[
\xymatrix{
{ A_\crys \otimes_{\sigma,\Ainf} M  } \ar[r]  & {   A_\crys/IA_\crys \otimes_{\sigma,\Ainf} M    } \ar[d]^\iso \\
& {  \Ainf / \xi\Ainf  \otimes_{\sigma,\Ainf} M } \ar[d]^\iso   \\
& {  \sigma^*M / \xi\sigma^*M } \ar[r]   &   {  \sigma^*M/N,}
}
\]
 is alternately characterized the $A_\crys$-submodule generated by all elements of the form $1\otimes \psi(m)$ and $a\otimes m$ with $m\in M$ and $a\in I$.  There is a unique $\sigma$-semi-linear map $\Phi_1: Q \to P$ whose effect on these generators is
\[
\Phi_1( 1 \otimes \psi(m)) = \frac{\sigma(\xi)}{p} \otimes m ,\qquad \Phi_1( a\otimes m) = \frac{\sigma(a)}{p} \otimes \phi(m).
\]
The following is a special case of  \cite[Theorem 1.5]{lau-BKF}.

\begin{theorem}[Lau]\label{thm:BKF to window}
The  construction $(M,\phi,\psi)\mapsto (P,Q,\Phi,\Phi_1)$ just given defines a functor from the category of Breuil-Kisin-Fargues modules to the category of windows. It is an equivalence of categories if $p>2$.
\end{theorem}


Given a window $(P,Q,\Phi,\Phi_1)$, define its \emph{\'etale realization}
\[
P_\et = \{ x\in Q : \Phi_1( x) = x \}
\]
as in \cite[\S 3]{lau-Galois}.
This is a torsion-free $\Z_p$-module equipped with a \emph{Hodge-Tate filtration}
\[
F_\HT(P_\et) \subset P_\et \otimes_{\Z_p} C,
\]
defined as the kernel of the $C$-linear extension of $P_\et \to Q/IP$.

Denote by $\mathrm{HTpair}$ the category of pairs $(T,W)$ in which $T$ is a torsion-free $\Z_p$-module,  and $W\subset T\otimes_{\Z_p} C$ is a subspace. 
Using the obvious notation for the categories of Breuil-Kisin-Fargues modules, $p$-divisible groups over $\co$, and windows, we  now have functors
\begin{equation}\label{functors}
\xymatrix{
{   \mbox{BKF-Mod} }  \ar[r]^{a} \ar[dr]_{d}  &   {    \mbox{Win}     }  \ar[d]_c   & { \mbox{$p$-DivGrp} } \ar[l]_{b}  \ar[dl]^{e} \\
& {   \mbox{HTpair}.   }
}
\end{equation}
Here  $a$ is given by Theorem  \ref{thm:BKF to window}, $b$ is given by Theorem  \ref{thm:group to window},  $c$ sends a window  to its \'etale realization, $d$ does the same for Breuil-Kisin-Fargues modules, and $e$ sends a $p$-divisible group over $\co$ to its $p$-adic Tate module endowed with its Hodge filtration.

\begin{remark}
It is not obvious from the definitions that  (\ref{functors}) commutes.
When $p>2$ the commutativity  is a byproduct of the following proof.
\end{remark}

\begin{proof}[Proof of Theorem \ref{thm:BKF}]
Assume that $p>2$.  In particular the functors of Theorems \ref{thm:group to window} and  \ref{thm:BKF to window} are equivalences of categories, and their composition gives the desired equivalence of categories between $p$-divisible groups over $\co$ and Breuil-Kisin-Fargues modules.

Suppose $G$ is a $p$-divisible group over $\co$, and let $(P,Q,\Phi,\Phi_1)$ and $(M,\phi,\psi)$ be its corresponding window and Breuil-Kisin-Fargues module. 
The isomorphisms
\[
\Omega(G^\vee)  \iso Q/IP  \iso N/\xi\sigma^*M 
\]
and 
\[
\Lie(G) \iso P/Q \iso \sigma^*M/N
\]
are clear  from the constructions of the functors of  Theorems \ref{thm:group to window} and  \ref{thm:BKF to window}.

The quotient map $\co \to k$ induces a ring homomorphism $\Ainf\to W(k)$ sending $\xi \mapsto  p$.  
It follows that there is a unique continuous extension to $A_\crys \to W(k)$ and, by (\ref{MtoP}),  canonical isomorphisms
\begin{equation}\label{crystal reduction}
W(k) \otimes_{A_\crys} P \iso W(k) \otimes_{\sigma,\Ainf} M \iso M_\crys.
\end{equation}
The functor of Theorem \ref{thm:group to window} is constructed in such a way that the leftmost $W(k)$-module  in (\ref{crystal reduction}) is identified with the value of the Dieudonn\'e crystal of $G_k$ at the divided power thickening $W(k) \to k$, which is just the usual covariant Dieudonn\'e module of $G_k$.

The  window of the constant $p$-divisible group $\Q_p/\Z_p$ over $\co$ consists of  
\[
P^0 = A_\crys \quad \mbox{and}\quad  Q^0=A_\crys
\]
 endowed with the operators $\Phi:P^0\to P^0$ and  $\Phi_1: Q^0 \to  P^0$   defined by 
\[
\Phi(x) = p  \sigma(x) \quad \mbox{and}\quad \Phi_1( x) = \sigma(x).
\]
In particular there is a canonical isomorphism 
$
Q^0/IP^0 \iso \co.
$

 The Breuil-Kisin-Fargues module of  $\Q_p/\Z_p$  consists of 
 \[
 M^0=\Ainf
 \]
  endowed with the operators
 \[ 
  \phi(x) = \xi \sigma(x) \quad \mbox{and}\quad \psi(x)=\sigma^{-1}(x).
  \]
The distinguished submodule $N^0\subset \sigma^*M^0$ defined in \S \ref{ss:BKF} is all of $\sigma^*M^0= \sigma^*\Ainf$, so is free of rank one generated by $1\otimes 1$.
Hence  there is a canonical isomorphism
 $
 N^0 / \xi \sigma^*M^0 \iso \co.
 $

From the equivalence of categories of Theorem \ref{thm:group to window} we obtain the commutative diagram
 \begin{equation}\label{etale match 1}
 \xymatrix{
 {T_p(G)} \ar@{=}[d] \ar[r]^{\HT}  &  { \Omega(G^\vee) } \ar@{=}[d]   \\
   { \Hom_{\mathrm{p-DivGrp} }(\Q_p/\Z_p, G) } \ar@{=}[d]  \ar[r] &  { \Hom_\co( \Omega(\mu_{p^\infty}) , \Omega(G^\vee) )} \ar@{=}[d]   \\
    {  \Hom_{\mathrm{Win}}(P^0,P) } \ar@{=}[d]\ar[r]    &    {  \Hom_\co( Q^0/IP^0 , Q/IP ) }  \ar@{=} [d]  \\
     {  P_\et  } \ar[r]  &     {  Q/IP.  }
     }
 \end{equation}
Similarly, from the equivalence of categories of Theorem \ref{thm:BKF to window} we obtain the commutative diagram
 \begin{equation}\label{etale match 2}
 \xymatrix{
 { M_\et } \ar@{=}[d] \ar[r]  &  { N/\xi\sigma^*M } \ar@{=}[d]   \\
   { \Hom_{\mathrm{BKF}}(M^0, M) } \ar@{=}[d]   \ar[r] &  { \Hom_\co(  N^0 / \xi \sigma^*M^0 ,  N/\xi\sigma^*M )} \ar@{=}[d]   \\
    {  \Hom_{\mathrm{Win}}(P^0,P) } \ar@{=}[d]\ar[r]    &    {  \Hom_\co( Q^0/IP^0 , Q/IP ) }  \ar@{=} [d]  \\
     {  P_\et  } \ar[r]  &     {  Q/IP.  }
     }
 \end{equation}
 Combining these gives  (\ref{BKF-HT}), completing the proof of Theorem \ref{thm:BKF}.

As a final comment, we note that the diagrams  (\ref{etale match 1}) and (\ref{etale match 2}) show that $P_\et$ and $M_\et$ are finitely generated $\Z_p$-modules,  and that 
  (\ref{functors})  commutes.
If we denote by $\mbox{FinHTpair} \subset  \mbox{HTpair}$ the full subcategory of pairs $(T,W)$ with $T$ of finite rank over $\Z_p$, we obtain a commutative diagram
\[
\xymatrix{
{   \mbox{BKF-Mod} }  \ar[r]^{a} \ar[dr]_{d}  &   {    \mbox{Win}     }  \ar[d]_c   & { \mbox{$p$-DivGrp} } \ar[l]_{b}  \ar[dl]^{e} \\
& {   \mbox{FinHTpair}   }
}
 \]
 in which the arrows $a$, $b$, and $e$ are equivalences of categories (the last one by Theorem \ref{thm:SW}).
 Hence all arrows are equivalences of categories.
\end{proof}


\section{Bounding the Hodge-Tate periods}
\label{s:core}


Let $G$ be a $p$-divisible group of height four and dimension two over $\co$, 
endowed  with an action $\Delta \to \End(G)$.

Throughout \S \ref{s:core} we do not require $p>2$. 
 Instead we allow $p$ to be arbitrary, but assume the conclusion of Theorem \ref{thm:BKF}.


\subsection{Hodge-Tate periods}
\label{ss:periods}


The  embeddings (\ref{embeddings}) determine a decomposition
\begin{equation}\label{partial omega}
\Omega(G^\vee) = \Omega_0(G^\vee)  \oplus  \Omega_1(G^\vee) ,
\end{equation}
in which each summand is free of rank one over $\co$, and  $\Z_{p^2}\subset \Delta$ acts on them through  $j_0$ and $j_1$, respectively.  The operator $\Pi$ maps each summand injectively into the other.
Applying $\otimes_\co k$ to (\ref{partial omega}) yields a decomposition
\[
\Omega(G_k^\vee) = \Omega_0(G_k^\vee) \oplus \Omega_1(G_k^\vee)
\]
into one dimensional $k$-vector spaces.

Composing the Hodge-Tate morphism (\ref{HT}) with the two projections yields two \emph{partial Hodge-Tate morphisms}
\[
T_p(G) \map{\HT_0}  \Omega_0(G^\vee)  ,\qquad T_p(G) \map{\HT_1}   \Omega_1(G^\vee)  .
\]
By fixing isomorphisms 
\begin{equation}\label{omega trivial}
\Omega_0(G^\vee) \iso \co ,\qquad \Omega_1(G^\vee)\iso \co,
\end{equation}
we view these as $\co$-valued $\Z_p$-linear functionals on $T_p(G)$.

\begin{lemma}
The $\Delta$-module $T_p(G)$ is free of rank $1$.
\end{lemma}

\begin{proof}
As $\Delta \otimes\Q_p$ is a division ring, its module $T_p(G)\otimes \Q_p$ is necessarily free.
Comparing $\Q_p$-dimensions shows that it is free of rank one, and hence  $T_p(G)$ is isomorphic to some   (left) $\Delta$-submodule of $\Delta\otimes \Q_p$.
As $\Delta$ admits a discrete valuation  \cite[Lemme II.1.4]{vigneras}   with uniformizer $\Pi$, every such submodule is principal and generated by a power of $\Pi$.
\end{proof}

Fix a $\Delta$-module generator $\lambda \in  T_p(G)$,  and define 
\[
\tau_0 = \frac{\HT_0(\Pi \lambda) }{\HT_0(\lambda) },
\qquad 
\tau_1 = \frac{\HT_1(\Pi \lambda) }{\HT_1(\lambda) }.
\]
These are the \emph{Hodge-Tate periods} of $G$.  In each fraction the numerator or denominator may vanish, but not simultaneously.
Thus the Hodge-Tate periods lie in   $\mathbb{P}^1(C)=C\cup\{\infty\}$.   They do not depend on the choice of (\ref{omega trivial}), 
but do  depend on the choice of generator $\lambda$.

\begin{proposition}\label{prop:period product}
The Hodge-Tate periods satisfy  $\tau_0\cdot \tau_1=p$.
\end{proposition}

\begin{proof}
The action of $\Pi$ on  $\Omega_0(G^\vee) \oplus \Omega_1(G^\vee) $ is  given by 
\[
(\omega_0,\omega_1) \mapsto ( s_0 \omega_1 , s_1\omega_0)
\]
for some $s_0,s_1\in \co$ satisfying $s_0 s_1=p$.   From the $\Delta$-linearity of the   Hodge-Tate morphism we deduce  first 
\[
 \HT_0(\Pi \lambda) = s_0 \cdot \HT_1( \lambda),\qquad 
 \HT_1(\Pi \lambda) = s_1 \cdot \HT_0( \lambda) ,
\]
and then 
\[
\tau_0\cdot \tau_1 = \frac{\HT_0(\Pi \lambda) }{\HT_0(\lambda)  } 
\cdot \frac{\HT_1(\Pi \lambda) }{\HT_1(\lambda)  }  = s_0\cdot s_1=p. \qedhere
\]
\end{proof}


\subsection{Reduction to the residue field}
\label{ss:reduction}


Let $G_k$ be the reduction of $G$ to the residue field $k=\co/\mathfrak{m}$,
and let $(D,F,V)$ be its covariant Dieudonn\'e module.

\begin{definition}
Let $H$ be the $p$-divisible group of a supersingular elliptic curve over $k$. In other words, $H$ is 
 the unique connected $p$-divisible group of height two and dimension one.
The reduction $G_k$ is said to be
\begin{enumerate}
\item
\emph{supersingular} if it is isogenous to  $H\times H$, 
\item
 \emph{superspecial} if it is isomorphic to $H\times H$.  
 \end{enumerate} 
\end{definition}

\begin{remark}
Our notions of supersingular and superspecial depend only on the $p$-divisible group $G_k$, and not on its $\Delta$-action.
This differs from the meaning of superspecial in some literature on Shimura curves, e.g.~\cite{KR}.
\end{remark}

The following proposition, which implies that the notion of superspecial depends only on the $p$-torsion subgroup scheme $G_k[p] \subset G_k$,  is well-known.  For lack of a  reference we provide the proof.

\begin{proposition}\label{prop:oort}
The reduction $G_k$  is supersingular, and the following are equivalent:
\begin{enumerate}

\item $G_k$ is superspecial,

\item 
 there is an isomorphism of group schemes $G_k[p] \iso H[p] \times H[p]$,

\item  $V^2D \subset pD$,

\item  $FD = VD$.
\end{enumerate}
\end{proposition}

\begin{proof}
The supersingularity of $G_k$ follows from the Dieudonn\'e-Manin classification of isocrystals: one can list all isogeny classes of $p$-divisible groups over $k$ of height four and dimension two, and the supersingular isogeny class is the only one whose endomorphism algebra contains a quaternion division algebra.  

The implication (1) $\implies$ (2) is trivial.  For the implication  (2) $\implies$ (3) it suffices to check that 
$V^2$ kills the Dieudonn\'e module of $H[p]$, which we leave to the reader.  
  
 Next we prove (3) $\implies$  (4).
 If $D' \subset D$ is any $W(k)$-lattice stable under both $F$ and $V$, then its corresponding $p$-divisible group $G'_k$ is isogenous  to $G_k$.  In particular it has dimension $2$, and hence 
\[
D'/VD'  \iso  \Lie(G'_k)
\]
is a $2$-dimensional $k$ vector space.  Applying this with $D'=D$ and $D'=VD$ shows that 
$D/V^2D$ has length $4$ as a $W(k)$-module.  On the other hand, $D/pD$ also has length $4$, proving the first implication in \begin{align*}
V^2 D \subset pD & \implies V^2 D = pD  \implies V D = F D.
\end{align*}

Finally, we prove (4) $\implies$ (1). 
Let $\alpha_p$ be the finite flat group scheme whose Dieudonn\'e module is the $W(k)$-module $k$, endowed with the operators $F=0$ and $V=0$.  If  $FD=VD$ then, using the self-duality of $\alpha_p$, we see that 
\begin{align*}
  \Hom(\alpha_p,G^\vee_k)  & \iso  \Hom(G_k[p],\alpha_p)   \\
& \iso  \Hom_k( D / (FD+VD)  , k) \\
& \iso  \Hom_k( D/VD, k ) \\
& \iso \Hom_k( \Lie(G) , k)
\end{align*}
is a $2$-dimensional $k$-vector space.    It follows from \cite[Theorem 2]{Oort75} that $G_k^\vee$ is superspecial, and hence so is $G_k$.
\end{proof}

Let $(M,\phi,\psi)$ be the  Breuil-Kisin-Fargues module  of $G$.  
The quotient
\[
M^\flat = M/pM
\]
is a free module over $\co^\flat \iso \Ainf/(p)$, endowed with operators $\phi ,\psi : M^\flat \to M^\flat$ satisfying $\phi \circ \psi = \varpi$.
Denote by $N^\flat = N/pN$ the image of 
\[
M^\flat \map{m\mapsto 1\otimes \psi(m)} \sigma^*M^\flat.
\]

 Each of our embeddings $j_0,j_1 : \Z_{p^2} \to \co$ determines a map
\[
\Z_{p^2}  \to \co/p\co \iso \co^\flat / \varpi \co^\flat,
\]
and  these two maps lift uniquely to  $j_0,j_1 : \Z_{p^2}  \to  \co^\flat.$
The action of $\Delta$ on $G$ determines an action  on  $M^\flat$, which induces  a decomposition  
\[
M^\flat=M^\flat_0\oplus M^\flat_1
\]
 analogous to (\ref{partial omega}).  It follows from the next proposition that each factor is free of rank two over $\co^\flat$.

 \begin{proposition}\
 \begin{enumerate}
 \item
 $D$  is free of rank one over  $\Delta \otimes_{\Z_p}W(k)$.
 \item
 $M$ is free of rank one over $\Delta\otimes_{\Z_p}\Ainf$.
 \end{enumerate}
  \end{proposition}

\begin{proof}
Reduce (\ref{embeddings}) to ring homomorphisms $j_0,j_1:\Z_{p^2} \to k$, and denote again by 
$j_0,j_1 : \Z_{p^2} \to W(k)$ the unique lifts. 
There is a   decomposition of $W(k)$-modules 
\[
D=D_0\oplus D_1
\]
in such a way that   $\Z_{p^2} \subset \Delta$ acts on the two summands via $j_0$ and $j_1$, respectively.
As in \cite[\S 1]{KR}, these summands  are free of rank $2$ over $W(k)$, and satisfy
\[
p D_0 \subsetneq VD_1 \subsetneq D_0 ,\qquad p D_1 \subsetneq VD_0 \subsetneq D_1.
\]
Moreover, either $\Pi D_0 = VD_0$ or $\Pi D_1=VD_1$ (or both).  

Without loss of generality, we may assume that  $\Pi D_0 = VD_0$, and hence
\[
p D_1 \subsetneq \Pi D_0 \subsetneq D_1.
\] 
Applying $\Pi$ to these inclusions shows that 
\[
p D_0 \subsetneq \Pi D_1 \subsetneq D_0.
\]
If we choose any $f_0\in D_0$ and $f_1\in D_1$ with nonzero images in  $D_0/\Pi D_1$ and $D_1/\Pi D_0$, 
respectively, then  $f_0 , f_1,  \Pi f_0 , \Pi f_1\in D$ reduce to a $k$-basis of $D/pD$.
Using Nakayama's lemma it is easy to see that $D$ is generated by 
$f_0+f_1$  as a $\Delta \otimes W(k)$-module, and the first claim of the proposition follows.

Theorem \ref{thm:BKF} gives us an isomorphism
\[
D/pD \iso  \sigma^*( M / \mathfrak{m} M)
\]
of $\Delta\otimes_{\Z_p}k$-modules, and from what was said above we deduce that $M/\mathfrak{m}M$ is free of rank one over $\Delta\otimes_{\Z_p} k$.  The second claim of the proposition follows easily from this and  Nakayama's lemma.
\end{proof}


\subsection{The case $\Pi \Omega(G_k^\vee)=0$}
\label{ss:first calculation}


We assume throughout \S \ref{ss:first calculation} that  
\[
\Pi  \Omega(G_k^\vee) =0 .
\]
We will analyze the structure of  $M^\flat$,  with its operators $\phi$ and $\psi$, 
and use this to bound the Hodge-Tate periods of $G$.
 The first step is to choose a  convenient basis.

\begin{lemma}\label{lem:coordinates 1}
There are $\co^\flat$-bases $e_0,f_0 \in M^\flat_0$ and $e_1,f_1\in M^\flat_1$ such that the operator $\Pi\in \Delta$ satisfies
\begin{equation}\label{pi good basis 1}
\Pi e_0=0,\qquad \Pi e_1 = 0,\quad \Pi f_0=e_1,\quad \Pi f_1 =e_0,
\end{equation}
and such that $\psi$ satisfies
\begin{equation*}
\psi(e_0)  = t_0  e_1  ,\quad 
\psi(e_1) = t_1 e_0  ,\quad 
\psi(f_0) =   e_1 + t_1  f_1 ,\quad
\psi(f_1) =   e_0 + t_0 f_0  
\end{equation*}
for  scalars $t_0, t_1 \in \co^\flat$ satisfying $\ord(t_0)>0$, $\ord(t_1)>0$, and 
\[
\ord(t_0) + \ord(t_1) = 1/p .
\]
\end{lemma}

\begin{proof}
As $M^\flat$ is free of rank one over $\Delta\otimes_{\Z_p} \co^\flat$, we may choose a basis such that (\ref{pi good basis 1}) holds, and the relation $\psi \circ \Pi = \Pi\circ \psi$  then implies 
\[
\psi(e_0)  = t_0  e_1  ,\quad 
\psi(e_1) = t_1 e_0  ,\quad 
\psi(f_0) =   u_1e_1 + t_1  f_1 ,\quad
\psi(f_1) =   u_0e_0 + t_0 f_0  
\]
 for  some $u_0,u_1 ,t_0,t_1 \in \co^\flat$.
The submodule $N^\flat \subset \sigma^*M^\flat$ is  generated by 
\begin{align*}
1\otimes \psi(e_0) & = t_0^p  \otimes  e_1  \\ 
1\otimes \psi(e_1) & = t_1^p \otimes  e_0 \\ 
1\otimes \psi(f_0) &= u_1^p \otimes e_1 +  t_1^p \otimes  f_1 \\ 
1\otimes \psi(f_1)&=  u_0^p \otimes  e_0 +  t_0^p \otimes   f_0.
\end{align*}

 Recall that $\mathfrak{m}^\flat \subset \co^\flat$ is the maximal ideal.
The first isomorphism in (\ref{BKF-dR})  identifies $\Omega(G_k^\vee)$  with the image of $N^\flat$ in $ (\sigma^*M^\flat ) / \mathfrak{m}^\flat  (\sigma^*M^\flat )$, and by hypothesis this $k$-vector space is annihilated by $\Pi$.
It  is easy to see from this that  $\ord(t_0)$ and $\ord(t_1)$ are positive.

Using Theorem \ref{thm:BKF}, we see that 
\[
\sigma^*M^\flat / N^\flat     \iso ( \sigma^*M/ N ) \otimes_\co \co/(p)   
 \iso \Lie(G)  \otimes_\co \co/(p)  
 \]
is free of rank two over $\co/(p) \iso \co^\flat/(\varpi)$.
On the other hand,  $\sigma^*M^\flat / N^\flat$ is isomorphic to the cokernel of the matrix
 \[
 \left(  \begin{matrix}
 t_1^p & & u_0^p \\
 & t_0^p & & u_1^p \\
 & & t_0^p  \\
 & & & t_1^p
 \end{matrix} \right) \in M_4(\co^\flat),
 \]
whose reduction  to $M_4(k)$ must therefore have rank $2$.  
This implies that $u_0$ and $u_1$  are units, and using 
 elementary row and column operations one sees that 
 \[
 \sigma^*M^\flat / N^\flat \iso \co^\flat/(t_0 t_1)^p  \oplus \co^\flat/(t_0 t_1)^p .
 \]
Hence  $(t_0 t_1)^p = (\varpi)$.
  Finally, having already seen that $u_0$ and $u_1$ are units,  an easy  calculation   shows that our basis elements may  be rescaled in order to make $u_0=1$ and $u_1=1$.
 \end{proof}

Fix a basis as in Lemma \ref{lem:coordinates 1}.   Theorem \ref{thm:BKF} identifies
\[
T_p(G) /pT(G) =   M^{ \psi =1} / p M^{\psi=1}  \subset  ( M^\flat )^{\psi =1},
\]
and the image of our fixed generator $\lambda\in T_p(G)$ has the form
\[
a_0 e_0+ a_1 e_1+b_0 f_0 + b_1 f_1 \in M^\flat
\]
for some coefficients $a_0,a_1,b_0,b_1\in \co^\flat$ satisfying
\begin{align}\label{tropical relations}
a_0^p & =  a_1 t_1^p +b_1 \\
a_1^p &= a_0 t_0^p + b_0   \nonumber   \\
b_0^p &= b_1 t_0^p    \nonumber  \\
b_1^p&= b_0 t_1^p.  \nonumber 
\end{align}

The first isomorphism of (\ref{BKF-dR}) identifies
\[
\Omega(G^\vee) / p \Omega(G^\vee)=      N / ( pN +  \xi \sigma^*M )    = N^\flat /\varpi \sigma^*M^\flat 
\]
with the direct summand of  $\sigma^* M^\flat / \varpi \sigma^* M^\flat$  generated by the reductions of 
\begin{align*}
1\otimes \psi(f_0)&=1\otimes e_1 + t_1^p\otimes f_1 \in \sigma^*M^\flat \\
1\otimes \psi(f_1)&=1\otimes e_0 + t_0^p \otimes f_0   \in \sigma^*M^\flat.
\end{align*}
If we use this basis to identify
\[
\Omega(G^\vee) / p \Omega(G^\vee)= N^\flat /\varpi \sigma^*M^\flat \iso \co^\flat/(\varpi) \oplus \co^\flat/(\varpi)
\]
then, again using Theorem  \ref{thm:BKF},  the partial Hodge-Tate morphisms
\begin{align*}
T_p(G)/pT_p(G) &  \map{\HT_0}  \Omega_0(G^\vee)  / p \Omega_0(G^\vee) \iso \co^\flat / (\varpi)  \\
T_p(G) /pT_p(G)  &  \map{\HT_1}  \Omega_1(G^\vee)  / p \Omega_1(G^\vee) \iso \co^\flat / (\varpi)  
\end{align*}
are given by
\begin{align}\label{HT swindle}
\HT_0(\lambda)   = a_1^p  &\qquad  \HT_0( \Pi \lambda)   = b_0^p  \\ 
\HT_1(\lambda)   = a_0^p   &\qquad  \HT_1(\Pi \lambda)   = b_1^p  . \nonumber
\end{align}

\begin{lemma}\label{lem:bord}
For $i\in \{0,1\}$, we have
\[
\ord(b_i) = \frac{ 1 }{ p^2-1} + \frac{ p\cdot  \ord(t_i)  }{p+1} .
\]
\end{lemma}

\begin{proof}
As  $\Pi  \lambda\in T_p(G)$ has nonzero image in  
\[
T_p(G) / p T_p(G)  \subset M^\flat,
\]
 we must have $b_0 e_1 + b_1 e_0 \neq 0.$
Therefore  one of $b_0$ and $b_1$ is nonzero.  The relations  (\ref{tropical relations}) then imply  first that both $b_0$ and $b_1$ are nonzero, and then that 
\[
b_i^{p^2-1}   =     (t_0t_1)^p  \cdot  t_i^{p(p-1)}  .
\]
The claim follows by applying $\ord$ to both sides of this equality.
\end{proof}

\begin{lemma}\label{lem:aord}
If we assume that 
\[
\frac{1}{p^2 (p-1)} < \ord(t_1),
\] 
 then 
\begin{align*}
\ord(a_0)  = \frac{ 1 }{ p(p^2-1)} + \frac{  \ord(t_1)  }{p+1}  ,\qquad
 \ord(a_1)   =  \frac{1}{p^2-1}  - \frac{  \ord(t_1)  }{p+1}.
\end{align*}
Of course there is a similar result if $t_1$ is replaced by $t_0$.
\end{lemma}

\begin{proof}
Recall the equality
$
a_0^p  =  a_1 t_1^p +b_1 
$
from (\ref{tropical relations}).  The only way this can hold is if  (at least) one of the three relations
\begin{itemize}
\item
$p\cdot \ord(a_0) = \ord(b_1 ) \le \ord(t_1^p a_1)$
\item
$p \cdot  \ord(a_0) = \ord(t_1^p a_1) \le \ord(  b_1)$
\item
$\ord(  b_1) = \ord(t_1^p a_1) \le p \cdot  \ord(a_0)$
\end{itemize}
is satisfied.  The second and third relations cannot be satisfied, as each implies 
\[
0 \le \ord(a_1) \le \ord( b_1)-p\cdot \ord(t_1) 
 =
 \frac{ 1 }{ p^2-1}
 -
 \frac{ p^2 \cdot \ord(t_1)  }{p+1} <0.
\] 
 Hence the first relation holds, and Lemma \ref{lem:bord} shows that
 \[
 p\cdot \ord(a_0) = \ord(b_1) = \frac{ 1 }{ p^2-1} + \frac{ p\cdot  \ord(t_1)  }{p+1}.
 \] 
 This proves the first equality.

 For the second equality, the relations (\ref{tropical relations})  imply 
\begin{align*}
 a_0^{p^2}   
&= a_1^p \cdot  (     t_1^p+     t_1  )^p   -  (t_0t_1)^p    a_0     \\
a_1^{p^2}     &= a_0^p \cdot  (    t_0^p+     t_0  )^p   -   (t_0t_1)^p    a_1      \nonumber.
\end{align*}
Using the second of these, along with  
\begin{align*}
 \ord\big(  a_0^p \cdot  (  t_0^p +   t_0)^p   \big) 
 &=   \ord(b_1)   +   p\cdot \ord(t_0)  \\
& = 
 \frac{ p^2 }{ p^2-1}  -  \frac{  p^2 \cdot \ord(t_1) }{p+1}    <  1 \le \ord\big( (t_0t_1)^pa_1 \big) ,
\end{align*}
we find  that 
\[
 \ord(a_1) =  \frac{ \ord\big(  a_0^p \cdot  (  t_0^p +   t_0)^p   \big) }{p^2} 
 = \frac{1}{p^2-1}  - \frac{  \ord(t_1)  }{p+1}.
\]
  \end{proof}

 Now we can prove the main result of this subsection.
 
 \begin{proposition}\label{prop:first case}
If we assume, as above, that $\Pi  \Omega(G_k^\vee) =0$ then 
\[
\frac{1}{p+1} < \ord(\tau_0) < \frac{p}{p+1}    \quad \mbox{and}\quad \frac{1}{p+1} < \ord(\tau_1) < \frac{p}{p+1}.
\]
\end{proposition}

\begin{proof}
First assume that 
\begin{equation}\label{param big}
\frac{1}{p^2(p-1)} < \ord(t_1).
\end{equation}
 The discussion leading to (\ref{HT swindle}) provides us with  an $\co$-module isomorphism
 \[
 \Omega_0(G^\vee) / p \Omega_0(G^\vee) \iso \co^\flat / (\varpi) \iso  \co/(p),
\]
and we fix  any  lift  to an isomorphism $\Omega_0(G^\vee) \iso \co$.

It is easy to see from Lemmas \ref{lem:bord} and \ref{lem:aord} that $\ord(a_1)$ and $\ord(b_0)$ lie in the open interval $(0,1/p)$, and so 
 $a_1^p$ and $b_0^p$ have nonzero images in $\co^\flat/(\varpi)$.
 By  (\ref{HT swindle}) these images agree with 
the images of  $\HT_0(\lambda)$ and $\HT_0(\Pi \lambda)$ under  
\[
\co \to \co/(p) \iso \co^\flat/(\varpi).
\]
Thus 
 \[
\ord(  \HT_0(\lambda) ) =   \ord(a_1^p)  = \frac{p}{p^2-1} - \frac{p \cdot \ord(t_1)}{p+1}  
 \]
 and
  \[
\ord(  \HT_0(\Pi \lambda) ) =   \ord(b_0^p)  = \frac{p}{p^2-1} + \frac{p^2 \cdot \ord(t_0)}{p+1} .
 \]
It follows  that 
\begin{align*}
\ord(\tau_0) & = \ord(  \HT_0(\Pi \lambda) ) - \ord(  \HT_0(\lambda) )   
  =   \frac{p}{p+1}  -   \frac{ (p-1) }{p+1} \cdot \ord(t_1^p),
\end{align*}
and so
\[
 \frac{1}{p+1} < \ord(\tau_0) < \frac{p}{p+1}.
\]
The analogous inequalities for $\ord(\tau_1)$ follow from the relation $\tau_0\tau_1=p$ of Proposition \ref{prop:period product}.

This proves Proposition \ref{prop:first case} under the assumption (\ref{param big}).  The proof when
\begin{equation}\label{param big 2}
\frac{1}{p^2(p-1)} < \ord(t_0)
\end{equation}
is  entirely similar.

Thus we are left to prove the claim under the assumption that both  (\ref{param big}) and (\ref{param big 2}) fail.
This assumption implies that
\[
\frac{1}{p}=\ord(t_0)+\ord(t_1) \le \frac{2}{p^2(p-1)},
\]
which implies that  $p=2$ and 
\[
\ord(t_0)= \frac{1}{4}=\ord(t_1).
\]
In particular,  Lemma \ref{lem:bord}  simplifies to 
\[
\ord(b_0)= \frac{1}{2}=\ord(b_1).
\]

Consider the equality $a_0^2  =  a_1 t_1^2 +b_1 $
of (\ref{tropical relations}).  As in the proof of Lemma \ref{lem:aord}, the only way this can hold is if  (at least) one of the relations
\begin{itemize}
\item
$ \ord(a_0) = 1/4$
\item
$\ord(a_1) = 0$ and $\ord(a_0)\ge 1/4$
\end{itemize}
holds.  Similarly, the equality $a_1^2  =  a_0 t_0^2 +b_0 $ implies that (at least) one of the relations
\begin{itemize}
\item
$ \ord(a_1) = 1/4$
\item
$\ord(a_0) = 0$ and $\ord(a_1)\ge 1/4$
\end{itemize}
holds.  Combining these shows that  $\ord(a_0) = 1/4$ and  $\ord(a_1) = 1/4$.

In particular, $a_1^p$ has nonzero image in $\co^\flat/(\varpi)$, and 
\[
\ord(\HT_0(\lambda)) = \ord(a_1^p) = \frac{1}{2}. 
\]
On the other hand, $b_0^p$ has trivial image in $\co^\flat/(\varpi)$, and so
\[
\ord(\HT_0( \Pi \lambda)) \ge 1.
\]
Therefore
\[
\ord(\tau_0) = \ord(\HT_0( \Pi \lambda))  -\ord(\HT_0(  \lambda))  \ge \frac{1}{2}.
\]
The same reasoning shows that $\ord(\tau_1)\ge 1/2$.
As  $\tau_0\tau_1=p$  by Proposition \ref{prop:period product}, we must therefore have
\[
\ord(\tau_0) = \frac{1}{2} = \ord(\tau_1),
\] 
 completing the proof of Proposition \ref{prop:first case}.
 \end{proof}


\subsection{The case $\Pi \Omega_1(G_k^\vee)\neq 0$}
\label{ss:second calculation}


We assume throughout \S \ref{ss:second calculation} that 
\[
\Pi  \Omega_1(G_k^\vee) \neq 0. 
\]
Once again, we will analyze the structure of $M^\flat=M/pM$, and use this to bound the Hodge-Tate periods of $G$.
As in \S \ref{ss:first calculation}, the first step is to choose a convenient basis for $M^\flat$.

\begin{lemma}\label{lem:coordinates 2}
There are $\co^\flat$-bases $e_0,f_0 \in M^\flat_0$ and $e_1,f_1\in M^\flat_1$ such that the operator $\Pi\in \Delta$ satisfies
\begin{equation}\label{pi good basis 2}
\Pi e_0=0,\qquad \Pi e_1 = 0,\quad \Pi f_0=e_1,\quad \Pi f_1 =e_0,
\end{equation}
and such that $\psi$ satisfies
\begin{equation}\label{psi good basis 2}
\psi(e_0)  =    e_1  ,\quad 
\psi(e_1) = t  e_0  ,\quad 
\psi(f_0) =    t   f_1 ,\quad
\psi(f_1) =   s  e_0 +  f_0  
\end{equation}
for  some scalars $s , t \in \co^\flat$ with $\ord(t)=1/p$.  
Moreover:
\begin{enumerate}
\item
For any such basis, $G_k$ is superspecial if and only if $\ord(s)>0$.  
\item
If $G_k$  is not superspecial such a basis can be found with $s=1$.
\end{enumerate}
\end{lemma}

\begin{proof}
Exactly as in the proof of Lemma \ref{lem:coordinates 1},  we may choose a basis such that (\ref{pi good basis 2}) holds, and such that 
\[
\psi(e_0)  = t_0  e_1  ,\quad 
\psi(e_1) = t_1 e_0  ,\quad 
\psi(f_0) =   u_1e_1 + t_1  f_1 ,\quad
\psi(f_1) =   u_0e_0 + t_0 f_0  
\]
for some $u_0,u_1,t_0,t_1 \in \co^\flat$ with $\ord(t_0)+\ord(t_1) = 1/p$. 

The $\Delta$-module $\Omega(G_k^\vee)$ is identified with the image of 
\[
N^\flat \to (\sigma^*M^\flat)  / \mathfrak{m}^\flat (\sigma^*M^\flat),
\]
and this identifies $\Omega_1(G_k^\vee)$ with the (one-dimensional) $k$-span of the vectors
\[
1\otimes \psi(e_1) =  t_1^p \otimes  e_0 ,\quad  1\otimes \psi(f_1) =  u_0^p \otimes e_0 + t_0^p \otimes  f_0
\]
in $(\sigma^*M_0^\flat)  / \mathfrak{m}^\flat (\sigma^*M^\flat)$. 
The assumption that $\Pi$ does not annihilate $\Omega_1(G_k^\vee)$ implies that $\ord(t_0)=0$, which  allows us to rescale our basis vectors to make $t_0=1$, and then add a multiple of $e_0$ to $f_0$ to make $u_1=0$.  Setting $t=t_1$ and $s=u_0$, the relations (\ref{psi good basis 2}) now hold.

It follows from  Proposition \ref{prop:oort} and Theorem \ref{thm:BKF}  that 
\begin{align*}
G_k \mbox{ is superspecial} &\iff V^2(D/pD) =0  \\
& \iff \psi^2 (M ^\flat /\mathfrak{m}^\flat M^\flat ) =0\\
&  \iff \ord(s) >0.
\end{align*}
 Finally, if $\ord(s)=0$ it is an easy exercise in linear algebra to see that the given basis elements can be rescaled to make $s=1$.
\end{proof}

As in \S \ref{ss:first calculation},  our fixed generator $\lambda\in T_p(G)$ determines an element 
\[
 a_0 e_0+ a_1 e_1+b_0 f_0 + b_1 f_1 \in M^\flat,
\]
where the coefficients $a_0,a_1,b_0,b_1\in \co^\flat$ satisfy
\begin{align}\label{tropical relations 2}
a_0^p & =  a_1 t^p +b_1 s^p \\
a_1^p &= a_0    \nonumber   \\
b_0^p &= b_1    \nonumber  \\
b_1^p&= b_0 t^p.  \nonumber 
\end{align}

As in \S \ref{ss:first calculation}, we may  identify
\[
\Omega(G^\vee) / p \Omega(G^\vee)=      N / ( pN +  \xi \sigma^*M )    = N^\flat /\varpi \sigma^*M^\flat 
\]
with the direct summand  of $\sigma^* M^\flat / \varpi \sigma^* M^\flat$  generated by the reductions of 
\begin{align*}
1\otimes \psi(e_0)  & =  1\otimes e_1 \in \sigma^*M^\flat \\
1\otimes \psi(f_1)  & =  s^p\otimes e_0 +1 \otimes f_0  \in \sigma^*M^\flat .
\end{align*}
If we use this basis to identify
\[
\Omega(G^\vee) / p \Omega(G^\vee)     =      N^\flat /\varpi \sigma^*M^\flat      \iso     \co^\flat/(\varpi) \oplus \co^\flat/(\varpi)
\]
then, using Theorem \ref{thm:BKF},  the partial Hodge-Tate morphisms
\begin{align*}
T_p(G)/p T_p(G) &  \map{\HT_0}  \Omega_0(G^\vee)  / p \Omega_0(G^\vee) \iso \co^\flat / (\varpi)  \\
T_p(G)/p T_p(G) &  \map{\HT_1}  \Omega_1(G^\vee)  / p \Omega_1(G^\vee) \iso \co^\flat / (\varpi)  
\end{align*}
satisfy
\begin{align}\label{HT swindle 2}
\HT_0(\lambda)   = a_0  &\qquad  \HT_0( \Pi \lambda)   = b_1  \\ 
\HT_1(\lambda)   = b_1   &\qquad  \HT_1(\Pi \lambda)   = 0  . \nonumber
\end{align}

\begin{lemma}\label{lem:newton 2}
We have
\[
\ord(b_0) = \frac{1}{p^2-1},\qquad \ord(b_1) = \frac{p}{p^2-1}.
\]
Moreover, 
\[
\ord(a_0) \ge  \frac{1}{p^2-1} ,\qquad \ord(a_1) \ge  \frac{1}{p(p^2-1)},
\]
and $G_k$ is superspecial if and only if one (equivalently, both) of these inequalities is strict.
\end{lemma}

\begin{proof}
Exactly as in the proof of Lemma \ref{lem:bord}, both $b_0$ and $b_1$ are nonzero.  The relations (\ref{tropical relations 2}) therefore imply that 
\[
b_0^{p^2-1} = t^p,
\]
from which the stated formulas for $\ord(b_0)$ and $\ord(b_1)=\ord(b_0^p)$ are clear.

The relations (\ref{tropical relations 2}) imply that $a_0$ is a root of $x^{p^2}-x t^{p^2}-b_1^ps^{p^2}$, and by examination of the  Newton polygon we see that
\[
\ord(a_0) \ge \frac{1}{p^2-1}
\]
with strict inequality if and only if $\ord(s)>0$.  Combining this with $a_1^p =a_0$ completes the proof.
\end{proof}

\begin{lemma}\label{lem:congruence}
If $G_k$ is not superspecial then 
\[
\varpi \left( a_0 / b_1  \right)^{p+1} \in (\co^\flat)^\times \quad \mbox{and} \quad  \varpi s^{p+1} /  t^p  \in (\co^\flat)^\times,
\]
and these units have  the same reduction to $k^\times$.
\end{lemma}

\begin{proof}
We have already noted that (\ref{tropical relations 2}) implies  $t^p=b_0^{p^2-1}$,  from which one easily deduces the equality
\[
\left( \frac{a_1}{b_0} \right) ^{p^2}  =   \frac{a_1}{b_0} + \frac{s^p}{  b_0^{p(p-1)}   } 
\]
in the fraction field of $\co^\flat$.
It  follows from this and Lemma \ref{lem:newton 2} that
\[
 \varpi^{ \frac{p}{p+1} } \left( \frac{a_1}{b_0} \right)^{p^2}    \quad
  \mbox{and } \quad  \left(  \frac{\varpi^{ \frac{1}{p+1} } s}{  b_0^{p-1}   }  \right)^p
\]
are units in $\co^\flat$ with the same reduction to $k^\times$, hence the same is true after raising both to the power  $(p+1)/p$.
The lemma follows easily from this and the relations (\ref{tropical relations 2}).
\end{proof}

\begin{proposition}\label{prop:second case} 
If we assume, as above, that $\Pi \Omega_1(G_k^\vee) \neq 0$ then 
\begin{equation}\label{north pole}
  \frac{p}{p+1} \le  \ord( \tau_1)  
\end{equation}
with strict inequality if and only if $G_k$ is superspecial.  Moreover, if equality holds then
\[
 - \frac{p}{  \tau_0^{p+1}}  \in \co^\times  \quad
  \mbox{and } \quad     
   \frac{ \varpi s^{p+1} }{  t^p }  \in (\co^\flat)^\times 
\]
 have the same reduction to $k^\times$.
\end{proposition}

\begin{proof}
Using (\ref{HT swindle 2}) and Lemma \ref{lem:newton 2}, we find that 
\[
\ord(\HT_0(\Pi \lambda) ) = \frac{p}{p^2-1},
\]
and that 
\[
\ord( \HT_0(\lambda) ) \ge  \frac{1}{p^2-1}
\]
with strict inequality if and only if $G_k$ is superspecial.  This implies that 
\[
\ord( \tau_0) = \ord(\HT_0(\Pi \lambda) )  - \ord( \HT_0(\lambda) )  \le \frac{1}{p+1}
\]
with strict inequality if and only if $G_k$ is superspecial.  
The inequality (\ref{north pole}) follows from this and the relation  $\tau_0\tau_1=p$
of Proposition \ref{prop:period product},  with strict inequality if and only if $G_k$ is superspecial.

Suppose that equality holds in (\ref{north pole}), so that $G_k$ is not superspecial.  
Choose an $\alpha \in \co^\flat$ satisfying $\alpha^{p^2-1} = \varpi$.  The construction of \S \ref{ss:tilt} determines an element 
$\alpha^\sharp \in \co$ whose image in 
$
\co / (p) \iso \co^\flat / (\varpi)
$
agrees with $\alpha$.

Combining the relations (\ref{HT swindle 2}) with Lemma \ref{lem:newton 2}  shows that 
\[
 \frac{ \mathrm{HT}_0(\Pi \lambda) }{ (\alpha^\sharp)^p }    \in \co^\times
\quad \mbox{and}\quad    \frac{ b_1 }{  \alpha^p }   \in (\co^\flat)^\times
\] 
have the same reduction to $k^\times$, as do 
\[
  \frac{ \mathrm{HT}_0( \lambda)  }{\alpha^\sharp}  \in \co^\times
\quad \mbox{and}\quad     \frac{ a_0 }{ \alpha}  \in (\co^\flat)^\times.
\] 
It  follows that 
\[
 \frac{ \tau_0}{ (\alpha^\sharp)^{  p-1 } }    \in \co^\times 
 \quad \mbox{and}\quad 
  \frac{ b_1   } {  a_0 \alpha^{p-1} }      \in   (  \co^\flat  )^\times
\]
have the same reduction to $k^\times$. 
Raising both to the power $p+1$ and applying Lemma \ref{lem:congruence} proves that
\[
\frac{\varpi^\sharp}{  \tau_0^{p+1}}  \in \co^\times  \quad
  \mbox{and } \quad     
   \frac{ \varpi s^{p+1} }{  t^p }  \in (\co^\flat)^\times 
\]
 have the same reduction to $k^\times$. Now apply  Lemma \ref{lem:sharp reduction}.
 \end{proof}


\section{The main results}
\label{s:main}


We now formulate and prove our main results on the Ekedahl-Oort stratification of the Hodge-Tate period domain $X$ defined by (\ref{drinfeld HT}).
Throughout \S \ref{s:main} we assume  that the conclusions of Theorem \ref{thm:BKF} hold. 
For example, it is enough to assume that $p>2$.


\subsection{The setup}
\label{ss:main setup}


Let $T$ be a free $\Delta$-module of rank one, and fix a generator $\lambda \in T$.  Use the embeddings (\ref{embeddings}) to decompose
\[
T\otimes_{\Z_p} C = T_{C,0} \oplus T_{C,1}
\]
as a direct sum of  $2$-dimensional $C$-subspaces, in such a way that the action of $\Z_{p^2} \subset \Delta$ on the summands is through $j_0$ and  $j_1$, respectively.
Using the projection maps to the two factors, we obtain injective $\Z_p$-linear maps
\[
q_0: T \to T_{C,0} ,\qquad  q_1:T \to T_{C,1} .
\]

To each  $\tau \in C\cup\{ \infty\}$ we associate the $\Delta$-stable plane  
\[
W_\tau  \subset T\otimes_{\Z_p} C
\]
  spanned by the two vectors
\[
 \tau  q_0(  \lambda)  -   q_0(\Pi \lambda)   \in  T_{C,0} ,\qquad 
 p  q_1(  \lambda)  -   \tau q_1(\Pi \lambda)   \in  T_{C,1}.
 \]
The construction   $\tau\mapsto W_\tau$ establishes a bijection 
\[
C\cup\{ \infty\} \iso X(C).
\]

\begin{remark}
It is not hard to see that  the above bijection $\mathbb{P}^1(C) \iso X(C)$ arises from an isomorphism of schemes over $\Q_{p^2}$.
The isomorphism cannot descend to $\Q_p$, for the simple reason that $X(\Q_p)=\emptyset$.
\end{remark}

For the rest of \S \ref{ss:main setup} and \S \ref{ss:main theorems} we hold 
  $\tau \in C\cup \{\infty\}$ fixed, and let $G$ be the  $p$-divisible group over $\co$ determined by the pair $(T,W_\tau)$.
Thus $G$ comes equipped with an action of $\Delta$, and   $\Delta$-linear identifications
\[
\xymatrix{
{      T_p(G)      }  \ar[rr]^{\HT} \ar@{=}[d]    &   &  {   \Omega(G^\vee) \otimes_\co C    }  \ar@{=}[d]  \\
{T} \ar[rr]   &    &  {  (T\otimes_{\Z_p} C) / W_\tau   .}  
}
\]
In the notation of \S \ref{ss:periods}, the Hodge-Tate periods of $G$ are
\begin{equation}\label{HT coord}
\tau_0=\tau \quad \mbox{and}  \quad \tau_1 = p/\tau.
\end{equation}


\subsection{Computing the reduction}
\label{ss:main theorems}


Let $G_k$ be the reduction of $G$ to the residue field $k=\co/\mathfrak{m}$, and let $(D,F,V)$ be its  covariant Dieudonn\'e module.
We will show how to compute the isomorphism class of $G_k[p]$   from the Hodge-Tate periods (\ref{HT coord}).

Let $\mathbb{D}=\Delta\otimes_{\Z_p} k$ with its natural action of $\Delta$ by left multiplication.
The embeddings (\ref{embeddings}) induce  a decomposition 
\[
\mathbb{D} = \mathbb{D}_0 \oplus \mathbb{D}_1
\]   
in which $\Z_{p^2}\subset \Delta$ acts on $\mathbb{D}_i$ through the composition of $j_i : \Z_{p^2} \to \co$ with the reduction map $\co\to k$.
Choose $k$-bases
\[
e_0 , f_0 \in \mathbb{D}_0 ,\qquad e_1,f_1 \in \mathbb{D}_1
\]
in such a way that  $\Pi \in \Delta$ acts as
\begin{equation}\label{pi relations}
\Pi e_0=0,\quad \Pi e_1 = 0,\quad \Pi f_0=e_1,\quad \Pi f_1 =e_0.
\end{equation}

\begin{theorem}\label{thm:stupidspecial stratum}
The inequalities 
\begin{equation}\label{equator}
\frac{1}{p+1} < \ord(\tau) < \frac{p}{p+1}
\end{equation}
hold if and only if  $\Pi \Omega(G_k^\vee)=0$.  When these conditions hold, 
there is a $\Delta$-linear isomorphism $D/pD\iso \mathbb{D}$ under which 
\[
 \begin{array}{llll}
{ F e_0 =  0,}  &  { Ff_0 =    e_1, }  & { Fe_1= 0, }    & {  Ff_1=  e_0,  }   \\
{ Ve_0 = 0, } & {  Vf_0 = e_1,}  & {  Ve_1=0,}  & { V f_1 = e_0 .}
\end{array}
 \]
\end{theorem}

\begin{proof}
If  $\Pi\Omega(G_k^\vee) \neq 0$ then either  $\Pi \Omega_1(G_k^\vee)\neq 0$ or $\Pi \Omega_0(G_k^\vee) \neq 0$.
In the first case Proposition \ref{prop:second case} implies 
\[
\frac{p}{p+1} \le \ord(\tau_1) .
\]
In the second case the  same proof, with indices $0$ and $1$ interchanged throughout, shows that 
\[
\frac{p}{p+1} \le \ord(\tau_0).
\]
In either case, these bounds imply that   (\ref{equator}) fails.

Now assume that $\Pi \Omega(G_k^\vee)=0$.  We have already proved in Proposition \ref{prop:first case} that  (\ref{equator}) holds,
and so it only remains to prove that $D/pD$ admits an isomorphism to $\mathbb{D}$ with the prescribed properties.

Let $e_0 , f_0 \in M_0^\flat$ and $e_1,f_1\in M_1^\flat$ be the bases of Lemma \ref{lem:coordinates 1}.  
Using the formula  for $\psi : M^\flat \to M^\flat$ prescribed in that lemma, and the relation $\phi \circ \psi =  \varpi$, one can write down an explicit formula for $\phi$, and then see that the induced operators  on the reduction  $M^\flat / \mathfrak{m}^\flat M^\flat$ are given by 
 \[
 \begin{array}{llll}
{ \phi (e_0) =  0,}  &  { \phi (f_0) =  u  e_1, }  & { \phi (e_1)= 0, }    & {  \phi (f_1) = u e_0,  }   \\
{ \psi (e_0) = 0, } & {  \psi (f_0) = e_1,}  & {  \psi (e_1)=0,}  & { \psi (f_1) = e_0 ,}
\end{array}
 \]
where  $u^{-1} \in k^\times$ is the reduction of $-t_0^p t_1^p / \varpi \in (\co^\flat)^\times$.

The images of $e_0,f_0,e_1,f_1$ under the bijection 
\[
M^\flat/\mathfrak{m}^\flat M^\flat \map{x\mapsto 1\otimes x} \sigma^* ( M^\flat/\mathfrak{m}^\flat M^\flat)
\iso M_\crys / pM_\crys
\iso D/pD
\]
provided by Theorem \ref{thm:BKF}  form a $k$-basis of $D/pD$, denoted the same way, satisfying the relations (\ref{pi relations}) 
and 
\[
 \begin{array}{llll}
{ F e_0 =  0,}  &  { Ff_0 =  u^{p}  e_1, }  & { Fe_1= 0, }    & {  Ff_1= u^{p} e_0,  }   \\
{ Ve_0 = 0, } & {  Vf_0 = e_1,}  & {  Ve_1=0,}  & { V f_1 = e_0 .}
\end{array}
 \]

It remains to prove that $u=1$.  The two embeddings (\ref{embeddings}) reduce to morphisms $j_0,j_1 : \Z_{p^2} \to k$, which then admit unique lifts to 
\[
j_0,j_1  : \Z_{p^2}  \to W(k).
\]
  This allows us to decompose 
$
D = D_0 \oplus D_1
$
as $W$-modules, where $\Z_{p^2}\subset \Delta$ acts on the two summands via $j_0$ and $j_1$, respectively.
Choose arbitrary lifts 
\[
\tilde{f}_0  \in D_0 ,\qquad \tilde{f}_1 \in D_1
\]
of $f_0$ and $f_1$, and then define 
\[
\tilde{e}_0 = \Pi \tilde{f}_1 \in D_0  ,\qquad \tilde{e}_1 = \Pi \tilde{f}_0 \in D_1.
\]
Using the fact that $\Pi$ and $V$ commute, we see that 
\[
\begin{array}{ccc}
{ V\tilde{e}_0  = pb_1 \tilde{e}_1 + p a_1 \tilde{f}_1  } & \quad  & { V \tilde{f}_0  = a_1 \tilde{e}_1 + p b_1 \tilde{f}_1 }  \\
{ V\tilde{e}_1  = pb_0 \tilde{e}_0 + p a_0 \tilde{f}_0  }  & \quad & { V \tilde{f}_1  = a_0 \tilde{e}_0 + p b_0 \tilde{f}_0 }
\end{array}
\]
for scalars 
\[
a_0  , a_1 \in 1+pW(k) ,\qquad b_0,b_1 \in W(k).
\]

Denote again by $\sigma : W(k) \to W(k)$ the lift of the Frobenius on $k$. 
Applying $F$ to the expressions for $V\tilde{e}_1$ and $V\tilde{f}_1$ results in 
\[
p \tilde{e}_1 = \sigma(pb_0) F\tilde{e}_0 + \sigma(pa_0) F \tilde{f}_0,\qquad
p \tilde{f}_1 = \sigma(a_0) F\tilde{e}_0 + \sigma(pb_0) F \tilde{f}_0 ,
\]
from which one deduces
\[
\big(  \sigma(a_0)^2 - p \sigma(b_0)^2 \big) \cdot F\tilde{f}_0 = \sigma(a_0) \tilde{e}_1 - p \sigma(b_0) \tilde{f}_1.
\]
Reducing this  modulo $p$ proves that $F f_0 = e_1$, and hence $u=1$.
\end{proof}

\begin{theorem}\label{thm:polar stratum 1}
The inequality 
\begin{equation}\label{stray bound}
 \ord(\tau) \le \frac{1}{p+1}
\end{equation}
holds if and only if  $\Pi \Omega_1(G_k^\vee)\neq 0$.  Moreover:
\begin{enumerate}
\item
If strict  inequality holds in (\ref{stray bound}),  there is a $\Delta$-linear isomorphism $D/pD\iso \mathbb{D}$ under which 
\[
\begin{array}{llll}
{ F e_0 = e_1 , } & {  F f_0 = 0,}  &  { F e_1=0 , }  & {  F f_1 =  f_0, } \\
{ V e_0 = e_1 ,}  &  { V f_0 = 0 , } &  {  V e_1=0, }  & {  V f_1 =  f_0 .}
\end{array}
\]

\item
If equality holds in (\ref{stray bound}),  there is a $\Delta$-linear isomorphism $D/pD\iso \mathbb{D}$ under which 
\begin{equation}\label{lower family}
\begin{array}{llll}
{ F e_0 =  u^p e_1 , } &  { F f_0 =  -u^pe_1   , } &  { F e_1=0 , } &  {  F f_1 =  u^pf_0  , } \\
{ V e_0 = e_1 ,  } &  {  V f_0 = 0 , }  &   {  V e_1=0, }   &  { V f_1 = e_0+f_0, }
\end{array}
\end{equation}
where $u$ is the image of  $-p /  \tau_0^{p+1}=-p /  \tau^{p+1}$   under  $\co^\times \to k^\times$.
\end{enumerate}
\end{theorem}

\begin{proof}
If  (\ref{stray bound}) holds then Theorem \ref{thm:stupidspecial stratum} implies that $\Pi \Omega(G_k^\vee)\neq 0$, and so 
either 
\[
\Pi \Omega_0(G_k^\vee) \neq 0 \quad \mbox{or} \quad \Pi \Omega_1(G_k^\vee)\neq 0.
\]
The first possibility cannot occur, as then the proof of Proposition \ref{prop:second case}, with the indices $0$ and $1$ reversed everywhere, would give the bound
\[
\frac{p}{p+1} \le \ord(\tau_0),
\]
contradicting (\ref{stray bound}).  
Conversely, if $\Pi \Omega_1(G_k^\vee) \neq 0$ then  (\ref{stray bound}) holds by Proposition \ref{prop:second case}.

Assume now that  (\ref{stray bound}) holds, and that $\Pi \Omega_1(G_k^\vee)\neq 0$.  Let 
 $e_0 , f_0 \in M_0^\flat$ and $e_1,f_1\in M_1^\flat$ be the bases of Lemma \ref{lem:coordinates 2}.
As in the proof of  Theorem \ref{thm:stupidspecial stratum}, the operator $\phi$ on $M^\flat$ can be computed from the formula for $\psi$ given in the lemma.
The induced operators on the reduction  $M^\flat / \mathfrak{m}^\flat  M^\flat$ are found to be 
 \[
 \begin{array}{llll}
{ \phi (e_0) = u e_1 ,}  &  { \phi (f_0) = -u v^p e_1 , }  & { \phi (e_1)= 0 , }    & {  \phi (f_1)= u f_0,  }   \\
{ \psi (e_0) = e_1, } & {  \psi (f_0) = 0 ,}  & {  \psi (e_1)=0 ,}  & { \psi (f_1) = v e_0 + f_0 ,}
\end{array}
 \]
 where $u\in k^\times$ is the reduction of $\varpi/t^p \in (\co^\flat)^\times$, and $v\in k$ is the reduction of $s\in \co^\flat$.
 By the final claim of  Lemma \ref{lem:coordinates 2}, we may further assume that 
 \[
 v= \begin{cases}
 0 & \mbox{if $G_k$ is superspecial} \\
 1 & \mbox{otherwise.}
 \end{cases}
 \]

Suppose  that strict inequality holds in (\ref{stray bound}).  Proposition \ref{prop:second case} tells us that $G_k$ is superspecial, and so $v=0$.
The images of $e_0,f_0,e_1,f_1$ under the bijection 
\[
M^\flat/\mathfrak{m}^\flat M^\flat \map{x\mapsto 1\otimes x} \sigma^* ( M^\flat/\mathfrak{m}^\flat M^\flat)
\iso M_\crys / p M_\crys
 \iso D/pD
\]
provided by Theorem \ref{thm:BKF}  form a $k$-basis of $D/pD$, denoted the same way, satisfying the relations (\ref{pi relations}) and 
 \[
 \begin{array}{llll}
{  Fe_0= u^p e_1 ,}  &  { F f_0 = 0 , }  & { Fe_1= 0 , }    & {  Ff_1= u^p f_0,  }   \\
{ Ve_0 = e_1, } & {  Vf_0 = 0 ,}  & {  Ve_1=0 ,}  & { Vf_1 =   f_0 . }
\end{array}
 \]
One can prove that $u=1$ by lifting the basis elements to $D$ and arguing exactly as in Theorem \ref{thm:stupidspecial stratum}.

 Suppose now that equality holds in (\ref{stray bound}).  
 Proposition \ref{prop:second case} implies  that $G_k$ is not superspecial,  so $v=1$, and that the reduction map $\co^\times \to k^\times$ sends 
$
-p /  \tau_0^{p+1}  \mapsto u .
$
 Arguing as in the previous paragraph, we obtain  a $k$-basis $e_0,f_0,e_1,f_1$ of $D/pD$ satisfying (\ref{pi relations}) and (\ref{lower family}), completing the proof.
\end{proof}

\begin{theorem}\label{thm:polar stratum 2}
The inequality
\begin{equation}\label{stray bound 2}
\frac{p}{p+1} \le \ord(\tau)
\end{equation}
holds if and only if $\Pi \Omega_0(G_k^\vee) \neq 0$. Moreover:
\begin{enumerate}
\item
If strict inequality holds in (\ref{stray bound 2}), there is a $\Delta$-linear isomorphism $D/pD\iso \mathbb{D}$ under which
\[
\begin{array}{llll}
{ F e_0 = 0 , }  &  {  F f_0 = f_1   ,}  &  {  F e_1= e_0 , }  &  {  F f_1 =  0 ,} \\
{ V e_0 = 0 , }  & {  V f_0 = f_1 , }  & { V e_1=    e_0   ,} & {  V f_1 = 0 }.
\end{array}
\]
\item
If equality holds in (\ref{stray bound 2}),    there is a $\Delta$-linear isomorphism $D/pD\iso \mathbb{D}$ under which 
\begin{equation}\label{upper family}
\begin{array}{llll}
{ F e_0 =  0 , }  & {  F f_0 =  u^p f_1 , }  & {  F e_1= u^p e_0, }  & {  F f_1 =  -u^p e_0 , } \\
{ V e_0 = 0 , }  & {  V f_0 = e_1+f_1 }  &  {  V e_1=e_0, }   &  {  V f_1 = 0.  }
\end{array}
\end{equation}
where $u$ is the image of $-p/\tau_1^{p+1}=-\tau^{p+1}/p^p$   under $\co^\times \to k^\times$.
\end{enumerate}
\end{theorem}

\begin{proof}
Recalling (\ref{HT coord}), the inequality (\ref{stray bound 2}) is equivalent to 
\[
\ord(\tau_1) \le \frac{1}{p+1}.
\]
Using this observation, the proof is identical to that of Theorem \ref{thm:polar stratum 1}, but with the indices $0$ and $1$ reversed everywhere.
\end{proof}

\begin{corollary}
The $p$-divisible group $G_k$ is superspecial if and only if 
\[
\ord(\tau) \not\in \left\{  \frac{1}{p+1} , \frac{p}{p+1} \right\}.
\]
Moreover, the superspecial locus of $X(C)$ is a union of three Ekedahl-Oort strata, characterized as follows:
\begin{enumerate}
\item
The subset of $X(C)$ defined by 
\[
\frac{1}{p+1} < \ord(\tau) < \frac{p}{p+1}
\]
is an Ekedahl-Oort stratum.  On this stratum  $\Pi \Omega(G_k^\vee)=0$.
\item
The subset of $X(C)$ defined by 
\[
\ord(\tau) < \frac{1}{p+1},
\]
is an Ekedahl-Oort stratum.  On this stratum  $\Pi \Omega_1(G_k^\vee)\neq 0$.
\item
The subset of $X(C)$ defined by 
\[
 \ord(\tau) > \frac{p}{p+1} .\
\]
is an Ekedahl-Oort stratum.  On this stratum  $\Pi \Omega_0(G_k^\vee) \neq 0$.
\end{enumerate}
\end{corollary}

\begin{proof}
Recall from Proposition \ref{prop:oort} that $G_k$ is superspecial if and only if $V^2$ annihilates $D/pD$.
Given this, all parts of the claim are clear from
 Theorems \ref{thm:stupidspecial stratum}, \ref{thm:polar stratum 1}, and  \ref{thm:polar stratum 2}.
\end{proof}

Now consider the locus of points 
\[
\left\{ \tau \in C : \ord(\tau) = \frac{1}{p+1} \right\}  \cup \left\{ \tau \in C : \ord(\tau) = \frac{p}{p+1} \right\}
\subset X(C)
\]
at which the corresponding $p$-divisible group does not have superspecial reduction.
This set is a union of infinitely many  Ekedahl-Oort strata.

\begin{corollary}
The fibers of the composition
\[
\left\{ \tau \in C : \ord(\tau) = \frac{1}{p+1} \right\} \map{ \tau \mapsto p/\tau^{p+1}} \co^\times \to k^\times 
\]
are Ekedahl-Oort strata, as are the fibers of the composition
\[
\left\{ \tau \in C : \ord(\tau) = \frac{p}{p+1} \right\} \map{ \tau \mapsto \tau^{p+1}/p^p} \co^\times \to k^\times .
\]
\end{corollary}

\begin{proof}
For each $u\in k^\times$ let   $F_u$ and $V_u$ be the operators on $\mathbb{D}$  defined by (\ref{lower family}).
Note that $V_u$ is actually independent of $u$. 
We claim that   the existence of a $\Delta$-linear isomorphism
\[
(\mathbb{D} , F_u ,V_u ) \map{\phi} (\mathbb{D} , F_{u'} ,V_{u'} )
\]
implies  $u=u'$.   To see this one  checks that the first relation in
\begin{equation}\label{intertwine}
\phi \circ V_u = V_{u'}\circ \phi ,\quad \phi \circ F_u = F_{u'}\circ \phi
\end{equation}
 implies that $\phi$ has the form
\[
\phi( e_0) = a e_0,\quad \phi(e_1) = ae_1 ,\quad \phi(f_0) = a f_0 ,\quad \phi(f_1) = a f_1 + b e_1
\]
for some $a\in \F_p$ and $b\in k$.
Using this, one checks that $\phi$ commutes with both $F_u$ and $F_{u'}$.  The second relation in (\ref{intertwine}) then implies that $F_u=F_{u'}$, and hence $u=u'$.

The same is true if we replace  the operators of (\ref{lower family}) with those of (\ref{upper family}), and 
so  the corollary  follows from  Theorems \ref{thm:polar stratum 1} and  \ref{thm:polar stratum 2}.
\end{proof}

\bibliographystyle{amsalpha}

\providecommand{\bysame}{\leavevmode\hbox to3em{\hrulefill}\thinspace}
\providecommand{\MR}{\relax\ifhmode\unskip\space\fi MR }
\providecommand{\MRhref}[2]{%
  \href{http://www.ams.org/mathscinet-getitem?mr=#1}{#2}
}
\providecommand{\href}[2]{#2}

\end{document}